\documentclass[12pt,letterpaper,english]{amsart}
\usepackage{amsmath,amsfonts,mathrsfs,amssymb}
\usepackage[letterpaper]{geometry}
\usepackage{setspace}
\usepackage{bbm}
\usepackage[bbgreekl]{mathbbol}
\DeclareMathSymbol\bbDelta  \mathord{bbold}{"01}
\usepackage{color}
\usepackage[all]{xy}
\usepackage{tikz}
\usetikzlibrary{arrows,decorations.pathmorphing,shapes}
\usepackage{etoolbox}
\usepackage{enumitem}
\usepackage{fullpage}
\AtBeginEnvironment{quote}{\singlespacing\small}

\def\X{\mathcal{X}}

\def\CC{\mathbb{C}}
\def\QQ{\mathbb{Q}}
\def\PP{\mathbb{P}}

\def\sm{\setminus}

\def\DD{\mathbb{D}}

\def\e{\epsilon}
\def\b{\bullet}
\def\ZZ{\mathbb{Z}}

\def\gr{\mathrm{Gr}}

\def\A{\mathcal{A}}
\def\ay{\mathbf{i}}
\def\H{\mathcal{H}}

\def\co{\mathcal{O}}
\def\SL{\mathrm{SL}}
\def\cs{\mathcal{S}}

\def\Y{\mathcal{Y}}
\def\uh{\underline{h}}
\def\DI{\Diamond}
\def\R{\mathcal{R}}
\def\x{\mathfrak{X}}
\def\XB{\overline{\X}}
\def\YB{\overline{\Y}}
\def\SB{\overline{\cs}}
\def\DD{\bbDelta}
\def\I{\mathrm{I}}
\def\II{\mathrm{II}}
\def\III{\mathrm{III}}
\def\IV{\mathrm{IV}}
\def\uht{\underline{\mathtt{h}}}
\def\ua{\underline{a}}
\def\uA{\underline{A}}
\def\ub{\underline{b}}
\def\uc{\underline{c}}
\def\fy{\mathfrak{Y}}

\theoremstyle{plain}
\newtheorem{thm}{Theorem}
\newtheorem{cor}[thm]{Corollary}
\newtheorem{prop}[thm]{Proposition}
\newtheorem{conj}[thm]{Conjecture}

\theoremstyle{definition}
\newtheorem{rem}[thm]{Remark}
\newtheorem{defn}[thm]{Definition}
\theoremstyle{definition}

\newtheorem{exa}[thm]{Example}
\newtheorem*{thx}{Acknowledgments}

\numberwithin{equation}{section}

\author{RJ Acu\~na}
\address{Department of Mathematics, Washington University in St.~Louis, 1 Brookings Drive, St.~Louis, MO 63130, USA}
\email{rjacuna@wustl.edu}

\author{Matt Kerr}
\address{Department of Mathematics, Washington University in St.~Louis, 1 Brookings Drive, St.~Louis, MO 63130, USA}
\email{matkerr@wustl.edu}

\begin{document}

\title{Hodge adjacency conditions for singularities}

\begin{abstract}
We prove compatibility relations between mixed Hodge numbers of $k$-Du Bois fibers in flat projective families and versal deformations of isolated $k$-Du Bois singularities.  These extend the notion of polarized relations in asymptotic Hodge theory beyond the normal-crossing boundary case, and we study combinatorial properties of the resulting weak polarized relations graphs.
\end{abstract}

\dedicatory{for Wilfried Schmid, who led the way in asymptotic Hodge theory}

\maketitle

In this note we introduce a theory of \emph{weak polarized relations} governing which other strata of a discriminant locus can be in the closure of a given stratum.  This theory applies to both families of varieties (e.g., over MMP-type compactifications of a moduli space) and deformations of singularities (initially just isolated ones).  The idea is that under an assumption that singularity types have something in common (both slc, or both $k$-log-canonical for some $k$), there should be a relation derived from Hodge-Deligne numbers or mixed spectra that governs whether one type can specialize to another.  

Here is a brief overview.  We first prove Theorem \ref{T1par} on 1-parameter degenerations $\mathfrak{X}\to \Delta$ whose singular fiber $X_0$ has $k$-Du Bois singularities, of lci type if $k>0$. It states that for all $r$, the specialization map $$H^r(X_0)\to H^r_{\lim}(X_t)$$
(always a morphism of mixed Hodge structure) is a bigraded isomorphism modulo $F^{k+1}$, which may be viewed as a ``common refinement'' of \cite[Thm.~1.6]{KL3} and \cite[Cor.~1.4]{FL}.  

Since the total space is \emph{not} assumed smooth in Theorem \ref{T1par}, it can then be used to ``probe'' the boundary in multiparameter degenerations.  This leads to our main Theorem \ref{T3}, which compares the mixed Hodge numbers of general fibers $X_0,X_1$ over adjacent strata $\Sigma_0,\Sigma_1$ in the discriminant locus of a flat projective family $\overline{\mathcal{X}}\to \overline{\mathcal{S}}$.  Specifically, if $\Sigma_0\subset \overline{\Sigma}_1$, and $X_0$ has $k$-Du Bois singularities (again of lci type if $k>0$), then 
\begin{equation}\label{eqI}
H^r(X_1)\preceq H^r(X_0)\;\;\;\;\text{mod}\;F^{k+1}\;(\forall r).
\end{equation}
Here ``$\preceq$'' is the notion of \emph{polarized relation} introduced in \cite{KPR}; it concerns only the Hodge-Deligne numbers $h^{p,q}(H)$ of a mixed Hodge structure $H$.  By definition, $H'\preceq H''$ means that there exist variations of Hodge structure over a product of punctured disks $\Delta^*\times\Delta^*$, with limit of type $H'$ along one axis and limit of type $H''$ at the center.  The \emph{weak} polarized relation \eqref{eqI} is saying that, despite the non-normal-crossing nature of the discriminant locus, the same relation governs the Hodge-Deligne numbers in $\gr_F^i$ for $0\leq i\leq k$.  An analogous statement relating the vanishing cohomology mixed Hodge structures of ``adjacent'' isolated $k$-Du Bois singularities is given in Corollary \ref{C6}.

By \cite[\S5]{KPR}, the polarized relations have a simple combinatorial description which is recalled at the beginning of \S\ref{S2}.  We use this to work out the weak polarized relations in the 0-Du Bois setting in varying levels of detail:  for weight $2$, a complete description in Proposition \ref{P10}; for weight $3$, a conjectural description; and finally, a general result for arbitrary weight in Proposition \ref{p19}.  This says that if the Hodge numbers beat a specific lower bound, then the weak polarized relations are given by the simplest inequalities possible and constitute a poset.  The lower bound is satisfied for projective hypersurfaces and double covers, which is reflected in the application to moduli of Horikawa surfaces in Example \ref{Ex5}.

Note that by \cite{KK}, the $0$-Du Bois case in this paper includes the semi-log-canonical singular fibers appearing in MMP-style compactifications.  For moduli of surfaces, we have seen a number of preprints taking for granted that polarized relations constrain the specializations of singular fibers within the discriminant locus.  (This is far from obvious since that locus is not a normal crossing divisor and the total space may be singular.)  Theorem \ref{T3} makes this rigorous and generalizes it to higher dimensions and singularity types.

\begin{thx}
We thank Mircea Musta\c t\u a for helpful correspondence and AIM for ideal collaborative working conditions during the workshop ``$k$-Du Bois and $k$-rational singularities''.  We also thank Radu Laza for a discussion and Sung Gi Park for pointing out a gap in the original proof of Theorem \ref{T1par}. The first author acknowledges the financial support of the Ann W.~and Spencer T.~Olin Chancellor's Graduate Fellowship.  The second author acknowledges support from NSF Grant DMS-2502708 and from the Simons Foundation.
\end{thx}

\section{Weak polarized relations}

Given a variation of Hodge structure over $(\Delta^*)^2$ with unipotent monodromies, questions about the relationship between its  limits on different strata of the ``discriminant locus'' arise:
\begin{itemize}
\item [(i)] What pairs of LMHS-types are possible on a coordinate axis and the origin (the notion of ``polarized relation'')?
\item [(ii)] What triples of LMHS-types are possible on the two coordinate axes together with the origin (the notion of ``polarizable triples'')?
\end{itemize}
These were the subject of a 2019 paper of Kerr, Pearlstein and Robles \cite{KPR}.  In the presence of a Mumford-Tate group condition, they showed that (i) is interesting and reduces to the identification of multivariable $\SL_2$-orbits (which can be done Lie-theoretically).  In the absence of a Mumford-Tate condition, (i) becomes relatively straightforward while (ii) remains a deep question in which Lie theory is unavoidable.

The goal here is to formulate Hodge-theoretic adjacency conditions for the much more general kinds of degenerations that arise, for example, in compactifications of moduli and miniversal deformations of singularities.  These conditions will of course be weaker than the ones in [op.~cit.], and merely necessary (not sufficient).  In particular, since the discriminant loci that arise are not locally normal-crossing, the results of [op.~cit.] do not apply --- the local monodromy is not even abelian, let alone described by a nilpotent cone.

For a projective variety $X$ with hyper-resolution $\e\colon \tilde{X}_{\b}\to X$, define $$\tilde{\Omega}^p_X:=\gr_F^p(R\e_*\Omega_{\tilde{X}_{\b}}^{\b})[p]$$ and consider the composition
\begin{equation*}
\Omega_X^p\overset{\phi^p}{\to}\tilde{\Omega}_X^p\overset{\psi^p}{\to}\DD_X(\tilde{\Omega}^{n-p}_X).
\end{equation*}
One says that $X$ is $k$-Du Bois [resp.~$k$-rational] if $\phi^p$ [resp.~$\psi^p\circ \phi^p$] is a quasi-isomorphism for $0\leq p\leq k$ \cite{FL}.  In many cases it is known by [op.~cit.] that $k$-rational implies $k$-Du Bois, and $k$-Du Bois implies $(k-1)$-rational.

\subsection{1-parameter degenerations}

Now consider a flat, projective degeneration $\x\to \Delta$ with reduced special fiber and $\pi\colon \x^* \to \Delta^*$ smooth (where $\x^*:=\x\sm X_0$).  We assume that $\x$ embeds as an open analytic subset of some projective algebraic $\XB$, with a flat projective morphism to a smooth curve.\footnote{Without loss of generality, $\XB\sm X_0$ can be assumed smooth, so that $\XB$ has the same singularity types as $\x$.}  Kerr and Laza had speculated that a generalization of the ``inversion of adjunction'' result of Schwede \cite{Sc} should hold, to the effect that $X_0$ $k$-Du Bois implies $\x$ has $k$-rational singularities.  The aim was to show the following generalization of \cite[Thm. 9.3]{KL1}, extending \cite[Thm.~1.6]{KL3} to non-smooth total spaces:

\begin{conj}\label{pKL1}
If $X_0$ has $k$-Du Bois singularities, then $(H^r(X_0))^{p,q}\cong (H^r_{\lim}(X_t))^{p,q}$ \textup{(}with trivial $T^{\text{ss}}$-action on the RHS\textup{)} for $0\leq p\leq k$ and all $q,r$.
\end{conj}

Inversion of adjunction has now been proved by Chen \cite[Thm.~1.2]{Ch} in the case where $\x$ has local complete intersection singularities,\footnote{It appears to be an open question whether this should generalize beyond the lci setting.} which allows us to prove this conjecture below for the lci case (Theorem \ref{T1par}).  To grease the skids, we review a few basic principles from \cite{KL1,KL3}.  Write $T=T_{\text{ss}}e^N$ for the monodromy of the local system $R^r\pi_*\ZZ_{\x^*}$, which underlies a VHS $\H=\H^r$ with Hodge numbers $\uh=(h^{p,r-p}(X_t))$.  The vanishing-cycle triangle for the constant mixed Hodge module $\QQ_{\X}^H$ induces an exact sequence of MHS
\begin{equation*}
\to H^r(X_0)\overset{\mathtt{sp}}{\to} H^r_{\lim}(X_t)\overset{\mathtt{can}}{\to}H^r_{\text{van}}(X_t)\overset{\delta}{\to}H^{r+1}(X_0)\overset{\mathtt{sp}}{\to}H^{r+1}_{\lim}(X_t)\to 
\end{equation*}
intertwining the action of $T_{\text{ss}}$ (by morphisms of MHS) and $N$ (by $(-1,-1)$ morphisms).  Both actions are trivial on $H^*(X_0)$.

By work of Schmid \cite{Sd}, the $\{h^{p,q}(H^r_{\lim})\}$ yield an admissible Hodge diamond in the sense of \S\ref{S2}.  Moreover, if we denote by $\tt{E}$ the semisimple operator acting on $(H^r_{\lim})^{p,q}$ by $p-q$, then the action of $N$ on $H^r_{\lim}$ extends to a representation of $\mathfrak{sl}_2$ intertwining $T_{\text{ss}}$ and $\mathtt{E}$.  An irreducible component (lying in a single eigenspace for both $T_{\text{ss}}$ and $\mathtt{E}$) of this representation, which takes the form of a direct sum of $\ell+1$ one-dimensional subspaces in $(H^r_{\lim})^{p_0+i,r-p_0-\ell+i}$ for $i=0,\ldots,\ell$, is called an \emph{$N$-string} of length $\ell$.  (We can also talk about $N$-strings in $H^r_{\text{van}}$, but they need not be centered about weight $p+q=r$.)  Since $T$ acts trivially on $H^*(X_0)$, the specialization map $\mathtt{sp}$ factors through $(H^r_{\lim})^T=\ker(N)\subset (H^r_{\lim})^{T_{\text{ss}}}$.

Now if $\x$ is smooth (or a rational homology manifold), $\QQ_{\x}^H$ is semisimple and the decomposition theorem applies; hence $\mathrm{im}(\mathtt{sp})=(H^r_{\lim})^T$, and $\mathrm{im}(\delta)$ is pure of weight $r+1$ and level $\leq r-1$ (cf.~\cite[Thm.~5.3, Prop.~5.5]{KL1}).  Denoting eigenvalues of $T_{\text{ss}}$ by $\xi$, an easy argument shows that the $N$-strings of $H^r_{\text{van}}$ are therefore:
\begin{itemize}[leftmargin=0.5cm]
\item (for $\xi\neq 1$) centered about weight $r$, and in $\ker(\delta)$; or
\item (for $\xi=1$) centered about weight $r+1$, and in $\ker(\delta)$ for length $>0$.
\end{itemize}
A first entry-point for $k$-Du Bois singularities is then given by:

\begin{prop}[{\cite[Thm.~1.6]{KL3}}]
If $\x$ is smooth, with special fiber $X_0$ $k$-log-canonical,\footnote{Since $X_0$ has hypersurface singularities, $k$-log-canonical and $k$-Du Bois are equivalent by \cite{JKSY}.} then $H^*_{\mathrm{van}}$ belongs to $F^{k+1}$.  Hence
\begin{itemize}[leftmargin=0.5cm]
\item the $N$-strings of $H^r_{\lim}$ of length $>0$ are in $F^{k+1}$ for $\xi\neq 1$ and in $F^k$ for $\xi=1$; and
\item $\mathtt{sp}^r\colon H^r(X_0)\to H^r_{\lim}(X_t)$ is an isomorphism mod $F^{k+1}$ and pure of weight $r$ mod $F^k$.
\end{itemize}
\end{prop}

In the case where $\x$ is not assumed smooth, but singularities of $X_0$ are at least of local complete intersection type, we now have the following:

\begin{thm}\label{T1par}
Suppose $X_0$ has $k$-Du Bois singularities, of lci type if $k>0$.\footnote{The $k=0$ case (with no lci assumption) is \cite[Thm.~9.3]{KL1}.}  Then $\mathtt{sp}^r$ is an isomorphism mod $F^{k+1}$ and pure of weight $r$ mod $F^k$.  In particular, $h^{p,q}(H^r(X_0))=h^{p,q}(H^r_{\lim})$ $\forall p\leq k$, with both zero when $p<k$ and $q\neq r-p$.
\end{thm}

An immediate consequence is that the Hodge-Deligne diagram for $H^r_{\lim}$ takes the form shown below, with the $\xi\neq 1$ part confined to the red square, and anything outside the red square determined by $H^r(X_0)$.

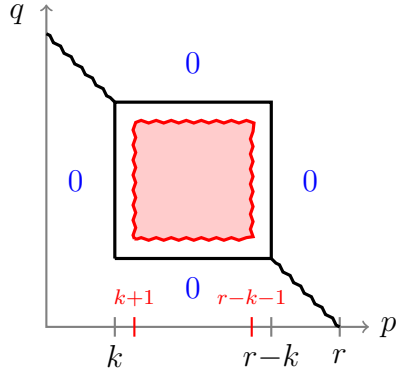
\[\begin{tikzpicture}[scale=1.3,decoration={zigzag,amplitude=0.5pt,segment length=2mm}]
\draw [gray,thick,<->] (0,3.3) -- (0,0) -- (3.3,0);
\node at (3.5,0) {$p$};
\node at (-0.3,3.2) {$q$};
\draw [gray,thick] (0.7,-0.1) -- (0.7,0.1);
\draw [gray,thick] (2.3,-0.1) -- (2.3,0.1);
\draw [gray,thick] (3,-0.1) -- (3,0.1);
\node at (3,-0.3) {$r$};
\draw [red,thick] (0.9,-0.1) -- (0.9,0.1);
\draw [red,thick] (2.1,-0.1) -- (2.1,0.1);
\node at (0.7,-0.3) {$k$};
\node at (2.3,-0.3) {$r{-}k$};
\node [red] at (0.9,0.3) {\tiny $k{+}1$};
\node [red] at (2.1,0.3) {\tiny $r{-}k{-}1$};
\draw [black,very thick,decorate] (0,3) -- (0.7,2.3);
\draw [black,very thick,decorate] (3,0) -- (2.3,0.7);
\draw [black,very thick] (0.7,0.7) -- (2.3,0.7) -- (2.3,2.3) -- (0.7,2.3) -- (0.7,0.7);
\node [blue] at (0.3,1.5) {$0$};
\node [blue] at (2.7,1.5) {$0$};
\node [blue] at (1.5,2.7) {$0$};
\node [blue] at (1.5,0.4) {$0$};
\draw [red,very thick,decorate,fill=red!20] (0.9,0.9) -- (0.9,2.1) -- (2.1,2.1) -- (2.1,0.9) -- (0.9,0.9);
\end{tikzpicture}\]

\begin{proof}[Proof of Theorem \ref{T1par}]
By assumption, $\x\to\Delta$ extends to $\XB\overset{\overline{\pi}}{\to} \SB$ flat projective algebraic, with $\cs$ smooth and $\XB\sm X_0$ smooth, and $X_0$ has lci singularities.  By \cite[Cor. 2]{Av}, it follows that $\XB$ also has lci singularities, and applying \cite[Thm.~1.2]{Ch} now gives that $\XB$ has \emph{$k$-rational} lci singularities.  Taking a log resolution yields a Cartesian square
\[\xymatrix{Y_0 \ar @{^(->} [r]^{\imath_Y} \ar [d]_{\beta_0} & \YB \ar [d]^{\beta} \\ X_0 \ar @{^(->}^{\imath_X} [r] & \XB}\]
in which $\YB$ is smooth and $Y_0\subset \YB$ is SNC.  (We also write $\mathfrak{Y}$ for the restriction of $\YB$ to $\Delta$.) By the proof of \cite[Thm.~3.24]{FL}, we have $\gr_F^p H^r(\XB)\hookrightarrow \gr_F^p H^r(\YB)$ for $p\leq k$ (and any $r$), which %together with exactness of \[0\to \gr_F^p H^*(\XB)\to \gr_F^p H^*(\YB)\oplus \gr_F^p H^*(X_0)\to \gr_F^p H^*(Y_0)\to 0\] implies injectivity of $\beta_0^*$ (hence $\mathtt{sp}_{\x}$) in \begin{equation}\label{eComp} \xymatrix{\gr_F^p H^r(X_0) \ar @{^(->} [r]^{\beta_0^*} \ar @/_2pc/ @{^(->} [rrr]_{\gr_F^p\mathtt{sp}^r_{\x}} & \gr_F^p H^r(Y_0) \ar [r]_{\cong\;\;\;\;\;}^{\gr_F^p\mathtt{sp}^r_{\YB}\;\;\;\;\;} & \gr^p_F(H^r_{\lim}(Y_t))^T  \ar @{^(->} [r] & \gr_F^p H^r_{\lim}(X_t)}. \end{equation} (Note that $X_t=Y_t$.) Continuing to take $p\leq k$, the ranks on either end of \eqref{eComp} are equal by \cite[Cor. 1.4]{FL}, forcing $\gr_F^p\mathtt{sp}_{\x}^r$ to be an isomorphism.  
implies exactness of 
\begin{equation}\label{e3r1}
0\to \gr_F^p H^r(\XB)\overset{\beta^*\oplus\imath_X^*}{\longrightarrow} \gr_F^p H^r(\YB)\oplus \gr_F^p H^r(X_0)\overset{\imath_Y^*-\beta_0^*}{\longrightarrow} \gr_F^p H^r(Y_0)\to 0.
\end{equation}
We claim that $\gr_F^p\mathtt{sp}^r_{\x}\colon \gr_F^pH^r(X_0)\to \gr_F^pH^r_{\lim}(X_t)$ is injective for $p\leq k$; since the ranks of $\gr_F^pH^r(X_0)$ and $\gr_F^pH_{\lim}^r(X_t)$ are equal in that range by \cite[Cor.~1.4]{FL}, this will force $\gr_F^p\mathtt{sp}_{\x}^r$ to be an isomorphism, giving the equalities of Hodge-Deligne numbers in the theorem.

Suppose then that $\xi\in \ker(\gr_F^p\mathtt{sp}_{\x}^r)$, and (noting that $H^r_{\lim}(Y_t)=H^r_{\lim}(X_t)$) pull it back via $\beta_0^*$ to $\tilde{\xi}\in\ker\{\gr_F^p\mathtt{sp}_{\fy}\colon \gr_F^pH^r(Y_0)\to\gr_F^pH^r_{\lim}(Y_t)\}$. Since $\fy$ is smooth, the local invariant cycle theorem applies to yield an element of $\gr_F^pH^r_{Y_0}(\fy)$ mapping to $\tilde{\xi}$; write $\tilde{\Xi}$ for its image in $\gr_F^pH^r(\YB)$, which satisfies $\imath_Y^*\tilde{\Xi}=\tilde{\xi}$.  Thus $(\tilde{\Xi},\xi)$ in the middle term of \eqref{e3r1} maps to zero in $\gr_F^pH^r(Y_0)$, and lifts to some $\Xi\in \gr_F^pH^r(\XB)$.  This $\Xi$ restricts to $\xi$ in $\gr_F^pH^r(X_0)$ and to zero in $\gr_F^pH^r(X_t)$ for $t\in\Delta^*$, since the latter restriction factors through $H^r(\YB)$ and thus through $H^r_{Y_0}(\fy)$.

Now because $X_0$ is $k$-Du Bois, we have $\gr_F^pH^r(X_0)\cong H^{r-p}(X_0,\Omega_{X_0}^p)$; and since $\XB$ is $k$-rational hence $k$-Du Bois we have $\gr_F^pH^r(\XB)\cong H^{r-p}(\XB,\Omega_{\XB}^p)$ as well.  These isomorphisms are compatible with restriction,\footnote{The log resolution of $(\XB,X_0)$ produces compatible hyperresolutions of $\XB$ and $\X_0$, whose respective $\phi^p$'s induce these isomorphisms.} and so the images $\Xi_t$ of $\Xi$ under the compositions
\begin{equation}\label{e3r2}
H^{r-p}(\XB,\Omega_{\XB}^p)\to \Gamma(\Delta,R^{r-p}\overline{\pi}_*\Omega^p_{\XB/\SB})\to H^{r-p}(X_t,\Omega_{X_t}^p)
\end{equation}
recover the restrictions of $\Xi$ from the last paragraph:  $\Xi_0=\xi$, and $\Xi_t=0$ for $t\neq 0$.  By \cite[Thm.~1.2]{FL} and the assumption that $X_0$ is $k$-Du Bois, $R^{r-p}\overline{\pi}_*\Omega^p_{\XB/\SB}$ is locally free (i.e., free on $\Delta$) for $p\leq k$.  It follows immediately that $\xi=0$, proving the claim.

For the (semi-)purity statement, any nonzero $\gamma\in (H^r_{\lim})^{p,q}$ (with $p\leq k$) comes from $\tilde{\gamma}\in (H^r(X_0))^{p,q}$, which gives $q\leq r-p$ and $\gamma\in \ker(N)$ (is at the bottom of an $N$-string).  If $q<r-p$, then $\gamma=N(\alpha)$ with $\alpha\in (H^r_{\lim})^{p+1,q+1}$.  But unless $p=k$, $\alpha\in \mathrm{im}(\mathtt{sp})$ (because $p+1\leq k$) hence $N(\alpha)=0$, a contradiction.
\end{proof}

\begin{rem}
In effect, Theorem \ref{T1par} refines the isomorphisms of $\gr_F^p$'s in \cite[Cor. 1.4]{FL} (replacing $H^*(X_t)$ by $H^*_{\lim}(X_t)$) ``in the $q$-direction'' by showing they are induced by the specialization map.
Since the LMHS is pure of weight $r$ outside the range $k\leq p\leq r-k$ for such a degeneration, the refinement is only ``interesting'' for $p=k$; but  this is still enough to produce nontrivial adjacency conditions in what follows.
\end{rem}

\begin{rem}
Sung Gi Park \cite[Thm.~11.1]{Pa} has independently proved a more general version of Theorem \ref{T1par}, by applying Saito's theory of mixed Hodge modules and results from \cite{PP}.  His version allows the general fiber to have $k$-rational lci singularities, whereas in ours the general fiber is smooth.
\end{rem}

\subsection{Multiparameter degenerations: normal-crossing discriminant}\label{S1.2}

Next we put ourselves in the situation of the diagram
\[\xymatrix{\x_s \ar @{^(->} [r] \ar [d] & \x \ar [d] & \x_{\Delta} \ar [d]\ar @{_(->} [l] \\ \{s\}\times \Delta \ar @{^(->} [r] & \underset{\tiny (s,t)}{\Delta \times \Delta} & \underset{\tiny u}{\Delta} \ar @{_(->} [l]}\]
with vertical morphisms projective, and smooth over $(\Delta^*)^2$.  Let $\H=\H^r$ denote the relative cohomology VHS on $(\Delta^*)^2$, with monodromies $T_s,T_t$. Since $\pi_1((\Delta^*)^2)\cong \ZZ\times\ZZ$ is abelian, $T_s$ and $T_t$ commute.

In fact, we are not concerned in this part about the fibers over the discriminant locus $(\{0\}\times \Delta) \cup (\Delta\times \{0\})$.  We are merely interested in the asymptotics of the restrictions of $\H$ to $\{s\}\times \Delta^*$ and the diagonal $\Delta^*$ (denoted $\DD^*$).  Writing $\psi_f$ for the nearby cycles functor for MHM (in the informal sense of ``limiting mixed Hodge structure at $f=0$''),
\begin{align}
\label{e1.2} H_{\lim}'&:=\psi_u(\H|_{\DD^*})&&\text{is a MHS ``at $(0,0)$'', while} \\
\label{e1.3} H_{\lim}(s)&:=\psi_t(\H|_{\{s\}\times \Delta^*})&&\text{is a MHS ``at $(s,0)$''.}
\end{align}
The weight filtrations are $W(N_s+N_t)[-r]_{\b}$ and $W(N_t)[-r]_{\b}$ respectively.  As $s$ varies, \eqref{e1.3} yields an admissible VMHS $\H_{\lim}(s)$ \cite[p. 76]{CK}, with limit at the origin satisfying\footnote{This is an immediate consequence of the existence of the local normal form [op.~cit., (2.8)] for a finite basechange of $\H^r$ killing semisimple parts of $T_s,T_t$.}
\begin{equation}\label{e1.4}
\psi_s \H_{\lim}(s)\cong H'_{\lim}.
\end{equation}
Write $\DI=\{h^{p,q}(H_{\lim}(s))\}$ and $\DI'=\{h^{p,q}(H'_{\lim})\}$ for the Hodge diamonds --- i.e., the functions on $\{0,1,\ldots,r\}^{\times 2}$ given by $(p,q)\mapsto h^{p,q}(\cdots)$.  

\begin{defn}\label{Dpol}
Whenever a pair of Hodge diamonds arise in the manner just described, we write $\DI\preceq \DI'$.  This defines the \emph{polarized relation} ``$\preceq$'' on the set of Hodge diamonds, where one should think of ``more degenerate'' as ``greater than''.  We shall also write $H_{\DI}\preceq H_{\DI'}$ for an ordered pair of MHS (or VMHS) whose Hodge diamonds are in polarized relation (viz., $H_{\lim}(s)\preceq H'_{\lim}$).
\end{defn}

See the beginning of \S\ref{S2} for an explicit description of the polarized relation.  Informally, the allowed degenerations of $\DI$ are induced by arbitrary combinations of degenerations of the primitive parts of the weight graded pieces (in weights between $0$ and $r$).  This is because:
\begin{itemize}[leftmargin=0.5cm]
\item by \eqref{e1.4}, the relative weight monodromy filtration exists and satisfies $M(W(N_t),N_s)=W(N_s+N_t)$, which means that $\DI$ can only degenerate in this way; and
\item the fact that the product of isometry\footnote{That is, automorphisms preserving the polarizing form.} groups of the $N_t$-primitive parts injects into the isometries of $H_{\lim}(s)$ is responsible for ``arbitrary combinations'' being possible. 
\end{itemize}
Describing the possible triples of Hodge diamonds for $\psi_t \H|_{\{s\}\times \Delta^*}$, $\psi_u\H_{\DD^*}$, and $\psi_s\H|_{\Delta^*\times\{t\}}$ is a much more difficult and interesting problem (cf.~\cite[\S8]{KPR}), but is not relevant here.

\subsection{Multiparameter degenerations: general discriminant}

Now consider a flat projective family $\XB\overset{\overline{\pi}}{\to} \SB$, with fibers smooth over the complement $\cs$ of the discriminant locus $\Sigma$.  (We do not need to assume that the base is smooth.)  Write $\Sigma=\sqcup \Sigma_{\alpha}$ for the decomposition into equisingular strata, and $\beta\colon \hat{\cs}\to \SB$ for a birational morphism with $\hat{\cs}$ smooth and $\hat{\Sigma}=\beta^{-1}(\Sigma)$ a NCD.  

Suppose that $\Sigma_1$ contains $\Sigma_0$ in its closure, and let $\mu\colon \Delta^2\to \hat{\cs}$ be a bi-disk embedding with $\mu((\Delta^*)^2)\subset \hat{\cs}\sm \hat{\Sigma}=\cs$ and $\beta(\mu(\Delta^*\times\{0\}))\subset \Sigma_1$, $\beta(\mu(0,0))\in \Sigma_0$.  Write $\rho=\beta\circ\mu$ and denote the $\rho$-base-change of $\XB\to\SB$ by $\x\to \Delta^2$, with fiberwise $r^{\text{th}}$ cohomology VHS $\H$.  In diagrammatic summary we have
\[\xymatrix@C=5pc{\XB \ar [d]^{\overline{\pi}} & \hat{\X} \ar [d]^{\hat{\pi}} \ar [l] & \x \ar [l] \ar [d] \\ \mspace{-40mu}\Sigma \subset \SB & \hat{\cs} \ar [l]^{\beta} & \Delta^2 \ar [l]^{\mu} \ar @/^2pc/ [ll]_{\rho}}\]
Denote the coordinates on $\Delta^2$ by $(s,t)$ and so forth, as in \S\ref{S1.2}.

Let $X_0$ be the fiber over $\rho(0,0)\in \Sigma_0$, and $X_1(s)$ the fiber over $\rho(s,0)\in \Sigma_1$ ($s\in \Delta^*$).  We are interested in the situation where $X_0$ has lci $k$-Du Bois singularities.  Since $\hat{\pi}$ remains flat projective and $\SB$ is smooth (hence lci), \cite{Av} gives that $\hat{\X}$, hence also fibers of $\hat{\pi}$, are lci in a neighborhood of $X_0$.  By semicontinuity of the singularity level in \cite[Thm.~3.37]{MP} (applied to $\mathcal{Z}:=\hat{\X}$) together with [op.~cit., Thm.~F] identifying singularity level with higher Du Bois type (using the lci assumption), we get that $X_1(s)$ has $k$-Du Bois lci singularities.  When $k=0$ we can use \cite[Cor.~4.2]{KS} to drop the lci assumption (and conclusion).

We are now in a position to apply Theorem \ref{T1par} to the restrictions of $\x\to \Delta^2$ to both the diagonal $\DD$ and to $\Delta\times\{s\}$.  This yields $\bar{F}^{\b}$-graded isomorphisms $$\text{$\gr_F^p H^r(X_0)\cong \gr_F^p \psi_u \H^r_{\DD}$\;\;\;\;and\;\;\;\;$\gr_F^p H^r(X_1(s))\cong\gr_F^p \psi_t\H^r$\;\;\;\;for $0\leq p\leq k$}$$ for each $r$.  Moreover, by \eqref{e1.4} we have $$\psi_u \H^r_{\DD}\cong \psi_s\psi_t\H^r$$ as mixed Hodge structures.  Finally, $\psi_s\psi_t\H^r$ is in polarized relation with $\psi_t\H$ in the sense of Definition \ref{Dpol}, written $$\psi_t\H^r\preceq\psi_s\psi_t\H^r.$$ Tracing through the identifications, we can state the conclusion as follows:

\begin{thm}\label{T3}
Given two strata $\Sigma_0,\Sigma_1$ in the discriminant locus of a flat projective family $\XB\to \SB$, with $\Sigma_0\subset \overline{\Sigma}_1$, and very general fibers $X_0,X_1$.  Assume that $X_0$ has $k$-Du Bois singularities, of lci type if $k>0$.  Then
\begin{equation}\label{eKL1}
H^r(X_1)\preceq H^r(X_0)\;\;\mod F^{k+1}\;\;(\forall r).
\end{equation}
\end{thm}

\begin{rem}\label{R4}
Once again, the mnemonic for this relation is that $A\preceq B$ says ``$B$ is more degenerate than $A$''.  More precisely, let $\uh=(h^0,h^1,\ldots,h^r)$, $h^p:=h^{p,r-p}(X)$, denote the Hodge numbers for a (smooth) fiber of $\X\to\cs$.  The $0^{\text{th}}$ through $k^{\text{th}}$ columns of the Hodge-Deligne diagram for $H^r(X_1)$ and $H^r(X_0)$ look like LMHSs of a VHS with Hodge numbers $\uh$ --- namely, $\psi_t\H$ and $\psi_u\H_{\DD}$.  These (L)MHS types are fixed outside the box $(p,q)\in B_k:=[k+1,r-k-1]^2$, but will depend on the choice of $\mu$ and $\beta$ inside $B_k$.  The Theorem says that the first (L)MHS type has to have an admissible 1-parameter degeneration to the second, without including them into any larger VMHS, for \emph{some} choice of the data in $B_k$.
\end{rem}

To our knowledge Theorem \ref{T3} is already new in the $k=0$ case.

\begin{exa}\label{Ex5}
Anthes \cite{An} computed the KSBA compactification $\overline{\mathcal{M}}_{2,4}$ of the moduli of Horikawa surfaces with $K_X^2=2$ and $\chi(\co_X)=4$ (arising as double covers of $\PP^2$ branched over an octic curve). The correspondence between incidence relations on boundary strata and polarized relations on the mixed Hodge structures on the singular fibers they parametrize is briefly mentioned without proof.  (The same could be said for the more recent \cite{CFPR}.) The above argument (with $k=0$) gives a simple proof that extends to all compactifications of moduli with slc (hence, by \cite{KK}, Du Bois) singular fibers.

More precisely, for a smooth such surface we have $p_g=h^{2,0}=3$ and $h^{1,1}_{\text{pr}}=37$. Denote by $\Sigma_{a,b}\subset \overline{\mathcal{M}}_{2,4}$ the stratum of the boundary whose singular fiber has $(H^2)^{0,0}=a$ and $(H^2)^{0,1}=b$ hence $(H^2)^{0,2}=3-a-b$.  Write $\DI_{a,b}$ for the corresponding Hodge diamond.  In the next section (Prop.~\ref{P10}), we will show that \eqref{eKL1} is given by $\DI_{a_1,b_1}\preceq \DI_{a_0,b_0}$ $\iff$ $a_1+b_1\leq a_0+b_0$ and $a_1\leq a_0$. It is shown in \cite{An} that, in fact, \emph{all} of these polarized relations are realized by the incidence relations on the $\{\Sigma_{a,b}\}$.
\end{exa}

\subsection{Deformations of isolated singularities}

Turning to versal deformations of an isolated $k$-Du Bois \emph{hypersurface} singularity $x_0$, we would like to get weak polarized relations on weighted spectra (of $x_0$ and any $x_1$ into which it deforms, also $k$-dB).  Here we focus on the Hodge-Deligne numbers of $H_{\text{van}}$ only, so that we don't have to worry about how test bi-disks ramify.  (With additional work one might hope to strengthen the relations by taking such things into account.)  The vanishing cohomologies here, say $V_0$ and $V_1$, are for \emph{generic} unramified disks passing through the given point in the discriminant locus $\Sigma$ (under $x_0$ or $x_1$), so that the ``total space'' over each is a smoothing (i.e.~with smooth total space).  These disks seem unlikely occur in the same bi-disk, but we can work around this as follows.

Assume that the versal deformation is relatively compactified to a sufficiently high degree projective family, so that along the generic disks we have $$L_i:=H_{\lim,i}\overset{\mathtt{can}_i}{\twoheadrightarrow} V_i.$$  Now a key point is that, by Theorem \ref{T1par}, the LMHS mod $F^{k+1}$ is independent of the disk along which we approach a singularity.  Hence the same thing is true for $F^{n-k}$ of the LMHS, where $n$ is the fiberwise dimension.  So now, if our singularity $x_0$ with vanishing cohomology $V_0$ (parametrized by $\Sigma_0$) deforms to a singularity $x_1$ with vanishing cohomology $V_1$ (parametrized by $\Sigma_1$), the argument above again gives $\psi_t \H\preceq \psi_u \H_{\Delta}$, whence $$F^{n-k}L_1\preceq F^{n-k}L_0.$$  Now we need to reduce modulo the possible images of the cohomologies of the singular fibers (which are contained in $W_n$), since these are the kernels of the $\mathtt{can}_i$ in the vanishing-cycle exact sequence.  This brings us to

\begin{cor}\label{C6}
Suppose an isolated $k$-Du Bois singularity $x_0$ in dimension $n$ deforms versally into an isolated singularity $x_1$, and let $V_0,V_1$ denote their vanishing cohomology MHS types.  Then 
\begin{equation}\label{eKL2}
F^{n-k}V_1 \preceq F^{n-k}V_0\;\;\mod W_{n}.
\end{equation}
\end{cor}

\begin{rem}
$\gr_F^{n-k}$ is the rightmost potentially nontrivial column in the Hodge-Deligne diagram for vanishing cohomology of $k$-dB singularities.
\end{rem}

\begin{exa}
For deformations of Du Bois surface singularities, this already has content:  the highest possible spectral number is $2$, and this can come with an accompanying weight of $3$ or $4$.  If the ``highest'' terms of $\tilde{\sigma}_{x_0}$ are $a_0[(2,4)]+b_0[(2,3)]$, and those of $\tilde{\sigma}_{x_1}$ are $a_1[(2,4)]+b_1[(2,3)]$, then we must have $a_1\leq a_0$ and $a_1+b_1\leq a_0+b_0$.  One would get something similar for $1$-dB 4-fold singularities.%would be good to write down actual singularities here!!!
\end{exa}

\section{Computing the relation \eqref{eKL1}}\label{S2}

Given a list of Hodge numbers $\uh=(h^0,h^1,\ldots,h^r)$, with $h^p=h^{r-p}$, denote by ``$\DI$'' a collection of nonnegative integers $\{\DI^{p,q}\}_{p,q\in [0,r]}$ representing Hodge-Deligne numbers of a LMHS type into which a VHS of type $\uh$ can degenerate.  The necessary and sufficient conditions are that $\DI^{q,p}=\DI^{p,q}=\DI^{r-q,r-p}=\DI^{r-p,r-q}$, and $\DI^{p,q}\leq \DI^{p+1,q+1}$ for $p+q<r$.  These are the \emph{admissible Hodge diamonds}.

With no constraint on the Mumford-Tate group, the polarized relation $\DI_1\preceq \DI_0$ is given by Definition \ref{Dpol}:  it means that a MHS of type $\DI_0$ can arise as an admissible degeneration of MHS of type $\DI_1$.  An \emph{explicit} description was worked out by Kerr, Pearlstein, and Robles \cite[Thm.~5.18]{KPR}.  Begin by writing $\DI_1$ as a sum of the form $\sum_{w=0}^r \sum_{a=0}^{r-w} P_w(-a)$, where $P_w^{p,w-p}=\DI^{p,w-p}-\DI^{p-1,w-p-1}$ are the primitive Hodge numbers in  weight $w$ and $P_w(-a)^{p,q}:=P_w^{p-a,q-a}$.  Let $\DI_{P_w}$ be any admissible Hodge diamond for $P_w$.  Then $$\textstyle\DI_0:=\sum_{w=0}^r\sum_{a=0}^{r-w}\DI_{P_w}(-a)$$ satisfies $\DI_1\preceq \DI_0$; and according to [op.~cit.], all polarized relations arise in this way.

While this looks fairly trivial, converting it into simple inequalities of the $\DI_i^{p,q}$ is not, and the relation turns out not to define a poset in general.  In this section we consider three basic combinatorial problems:
\begin{enumerate}%[leftmargin=0.7cm]
\item give a numerical characterization of the full polarized relations (for given $\uh$);
\item numerically describe the weak polarized relations \eqref{eKL1} (for given $\uh$ and $k$); and
\item determine for which $\uh$ each of these define (a) a linear order, (b) a partial order, or (c) fail to yield a poset.
\end{enumerate}
We will give a complete solution in weight 2, partial results in weight 3, and a general result for weak relations assuming certain lower bounds on middle Hodge numbers.

\begin{rem}
(i) These relations can be encoded in directed graphs $\R(\uh)$ and $\R_k(\uh)$, with admissible Hodge diamonds (or diamonds mod $F^{k+1}$) as vertices and $\preceq$ (or $\preceq$ mod $F^{k+1}$) as arrows.  By Remark \ref{R4}, $\R_k(\uh)$ is a quotient of $\R(\uh)$ in the sense that it is surjective on vertices and we draw an arrow from one vertex to another if there is an arrow between any pair of vertices in their preimages.  It encodes the polarized relations modulo the ``black box'' $B_k=[k+1,r-k-1]^2$.

(ii) Note that the arrows do not in general obey transitivity, and so ``all arrows must be drawn'' unless it is explicitly stated that the graph describes a poset.

(iii) For LMHSs arising from $k$-dB singular fibers, we have to consider the subgraph $\R^{\circ}_k(\uh)\subseteq \R_k(\uh)$ whose vertices correspond to MHSs which are pure outside $B_{k-1}=[k,r-k]^2$.  One might expect that this can be a poset even if $\R_k(\uh)$ is not, and $\R_k(\uh)$ could be a poset even if $\R(\uh)$ is not.  Notice that in $\R^{\circ}_k(\uh)$, all of the action is taking place on the ``frontier'' of $B_{k-1}$, i.e. $B_{k-1}\sm B_k$.
\end{rem}

In the setting of Theorem \ref{T3}, we can associate to each $k$-dB stratum a well-defined ``Hodge diamond modulo $B_k$'', denoted $\widetilde{\DI}\in \R_k^{\circ}(\uh)$, as follows.  We take $\widetilde{\DI}^{p,q}:=h^{p,q}(H^r(X))$ for $0\leq p\leq k$ ($q$ arbitrary), where $X$ is a fiber over the stratum, and then extend by symmetry to all $(p,q)\notin B_k$.

\begin{cor}
For any ``adjacent pair'' $\Sigma_0\subset \overline{\Sigma}_1$ of $k$-dB strata as in Theorem \ref{T3}, the Hodge diamonds satisfy $\widetilde{\DI}_1\preceq \widetilde{\DI}_0$ in $\R_k^{\circ}(\uh)$.
\end{cor}

\subsection{Degenerations of weight 2 Hodge structures}
We begin with the simplest nontrivial case, where $r=2$.  Then we have $\uh=(h^0,h^1,h^0)$, and there is only one relations graph $\R(\uh)=\R_0(\uh)=\R^{\circ}_0(\uh)$ as the Hodge numbers on the ``frontier'' $B_{-1}\sm B_0$ already determine the Hodge number $\DI^{1,1}_i$ in $B_0$.  In other words, for the middle cohomology of surfaces there is no difference between the weak and strict relations.  Write $$\text{$\DI_i^{0,0}=\DI_i^{2,2}=a_i$ \;\;\;\;and\;\;\;\; $\DI_i^{0,1}=\DI_i^{1,0}=\DI_i^{1,2}=\DI_i^{1,2}=b_i$;}$$ then we also have $\DI_i^{0,2}=\DI_i^{2,0}=h^0-a_i-b_i$ and $\DI_i^{1,1}=h^1-2b_i$.  These are admissible Hodge diamonds (i.e.~vertices of $\R(\uh)$) iff 
\begin{equation}
a_i+2b_i\leq h^1\;\;\;\text{and}\;\;\;a_i+b_i\leq h^0,
\end{equation}
since we must have $\DI_i^{0,0}\leq \DI_i^{1,1}$ (and all $\DI_i^{p,q}\geq 0$).

\begin{prop}\label{P10}
Suppose $\DI_0$ and $\DI_1$ are admissible. We have $\DI_1\preceq \DI_0$ $\iff$ 
\begin{equation}\label{e22}
a_1\leq a_0,\;\;a_1+b_1\leq a_0+b_0,\;\;\text{and}\;\;a_0+b_0\leq h^1-b_1.
\end{equation}
\end{prop}

\begin{proof}
Write 
\begin{itemize}[leftmargin=0.5cm]
\item ``I'' for the degeneration sending a weight 2 HS of type $(1,1,1)$ to a MHS with $h^{0,0}=h^{1,1}=h^{2,2}=1$;
\item ``II'' for the degeneration sending a weight 2 HS of type $(1,2,1)$ to a MHS with $h^{1,0}=h^{0,1}=h^{2,1}=h^{1,2}=1$; and
\item ``III'' for the degeneration sending a MHS with $h^{1,0}=h^{0,1}=h^{2,1}=h^{1,2}=1$ to a MHS with $h^{0,0}=h^{2,2}=1$ and $h^{1,1}=2$.
\end{itemize}
Any degeneration of MHS from $\DI_1$ to $\DI_0$ is a sum of $k$ copies of ``I'', $\ell$ copies of ``II'', and $m$ copies of ``III''.  The first two can only use the primitive part of the $\gr^W_2$ of $\DI_1$.  Therefore such a degeneration is possible iff $m\leq b_1$, $k+\ell\leq h^0-a_1-b_1$, and $k+2\ell\leq h^1-a_1-2b_1 (=\DI_{1,\text{pr}}^{1,1})$.  It yields $a_0=a_1+m+k$ and $a_0+b_0=a_1+b_1+k+\ell$, so also $b_0=b_1+\ell-m$ and $a_0+2b_0=a_1+2b_1+k+2\ell-m$.

($\implies$) On the one hand, this gives $a_0-a_1=m+k\geq 0$, $a_0+b_0-(a_1+b_1)=k+\ell\geq 0$, and $a_0+b_0=a_0+2b_0-b_0= a_1+2b_1+k+2\ell-m-b_0\leq h^1-m-b_0=h^1-b_1-\ell\leq h^1-b_1$.

($\impliedby$) For the converse, first suppose $b_1\leq b_0$ in addition to \eqref{e22}.  Then we can take $m=0$, $\ell=b_0-b_1$, and $k=a_0-a_1$.  If instead $b_0\leq b_1$, we can take $\ell=0$, $m=b_1-b_0$, and $k=a_0+b_0-a_1-b_1$.  In either case, we have $k,\ell,m\geq 0$ and the degeneration exists.
\end{proof}

\begin{cor}
The polarized relation defines a poset if and only one of the following hold:
\begin{itemize}[leftmargin=0.5cm]
\item $h^0$ or $h^1$ is $1$ \textup{(}in which case the order is linear\textup{)};
\item $h^0\geq 2$ and $h^1\geq 2h^0-1$ \textup{(}where the order is linear only for $\uh=(2,3,2)$\textup{)}.
\end{itemize}
\end{cor}

\begin{proof}
Transitivity fails when a type III adds to $\DI_{\text{pr}}^{1,1}$; this can change the primitive $\gr^W_2$ from (say) $(1,0,1)$ to $(1,1,1)$, permitting a type I in a subsequent degeneration.  But the type I cannot be done together with the type III as a single degeneration.

To quantify this, use \eqref{e22} and write $\DI_{a,b}$ as in Example \ref{Ex5}.  If $h^1\leq 2h^0-2$, then for $h^1$ even $\DI_{0,\frac{h^1}{2}}\preceq \DI_{\frac{h^1}{2},0}\preceq \DI_{1+\frac{h^1}{2},0}$ and for $h^1$ odd $\DI_{1,\frac{h^1-1}{2}}\preceq \DI_{1+\frac{h^1-1}{2},0}\preceq \DI_{2+\frac{h^1-1}{2},0}$, with transitivity failing in both.

For the converse, assume $\{\Diamond_i\}_{i=0}^2$ are admissible, with
$\Diamond_2 \preceq \Diamond_1$ and $\Diamond_1\preceq
\Diamond_0$, and $h^1\geq 2h^0-1$. By Proposition \ref{P10}, we have $a_2\leq a_1\leq a_0$, $a_2+b_2\leq a_1+b_1\leq a_0+b_0$, $a_1+b_1\leq h^1-b_2$, and $a_0+b_0 \leq h^1-b_1$.  To show $\Diamond_2 \preceq \Diamond_0$, it remains to check that $a_0+b_0\leq h^1-b_2$.  Since $h^1\geq a_i+2b_i$ and $h^0\geq a_i+b_i$ for each $i$ (admissibility), we have $h^1\geq 2h^0-1 \geq a_0+b_0 +a_2+b_2 -1$ hence $h^1-b_2 \geq a_0+b_0+a_2-1$.  

If $a_2\geq 1$ or the last inequality is strict, then $h^1-b_2\geq a_0+b_0$ and we are done.  Suppose then that $a_2=0$ and $h^1-b_2=a_0+b_0-1$.  Then in $h^1\geq 2h^0-1\geq h^0+b_2-1\geq a_0+b_0+b_2-1$ all inequalities are equalities, so that $b_2=h^0=a_0+b_0$.  By admissibility $h^1\geq 2b_2$, so
$h^1-b_2 \geq b_2=a_0+b_0$.  This contradiction completes the proof.
\end{proof}

\begin{rem}
Consistently with its correspondence to closure order in Example \ref{Ex5}, $\R(2,37,2)$ is a poset.  Its graph looks like a $3\times 3$ part of a triangular tiling.	
\end{rem}

\subsection{Degenerations of weight 3 Hodge structures}

When $r=3$, with $\uh=(h^0,h^1,h^1,h^0)$, the relations graph $\R_0(\uh)=\R^{\circ}_0(\uh)$ is a nontrivial quotient of $\R(\uh)$.  The Hodge diamonds are determined by
$$\DI^{0,0}=\DI^{3,3}=a,\;\;\DI^{1,0}=\DI^{0,1}=\DI^{2,3}=\DI^{3,2}=b,$$
$$\DI^{2,0}=\DI^{0,2}=\DI^{1,3}=\DI^{3,1}=c,\;\;\text{and}\;\DI^{1,1}=\DI^{2,2}=d+a$$
with $a,b,c,d\geq 0$, $a+b+c\leq h^0$, and $a+b+c+d\leq h^1$, so that $\DI^{1,2}=\DI^{2,1}=h^1-a-b-c-d$ and $\DI^{3,0}=\DI^{0,3}=h^0-a-b-c$.  

\begin{exa}[{\cite[Ex. 5.22]{KPR}}]
Here is how things look in the special (Calabi-Yau) case where $h^0=1$.  Instead of $\DI$, write $\I_d,\II_d,\III_d,\IV_d$ for diamonds, where $d\leq h^1$ and
\begin{itemize}[leftmargin=0.5cm]
\item ``$\I$'' means that $a=b=c=0$;
\item ``$\II$'' means that $a=b=0$ and $c=1$;
\item ``$\III$'' means that $a=c=0$ and $b=1$; and
\item ``$\IV$'' means that $b=c=0$ and $a=1$.	
\end{itemize}
(The indexing of type $\IV$ is different than in [loc.~cit.]; clearly $d<h^1$ for $\II,\III,\IV$.) Then the polarized relation is completely described by:
$$\mathrm{K}_{d_1}\preceq \mathrm{K}_{d_0}\;\iff\;d_1\leq d_0\;\;\;\;(\mathrm{K}=\I,\II,\III,\IV),\;\;\;\;\;\;\;\;\III_{d_1}\preceq \IV_{d_0}\iff d_1+1\leq d_0,$$
$$\II_{d_1}\preceq \IV_{d_0}\iff 1\leq d_1\leq d_0,\;\;\;\;\;\;\II_d\preceq \III_{d_0}\iff 2\leq d_1\leq d_0+2,$$
$$\I_{d_1}\preceq \III_{d_0}\iff d_1\leq d_0,h^1-2,\;\;\;\;\;\text{and}\;\;\;\;\;\;\I_{d_1}\preceq \II_{d_0},\IV_{d_0}\iff d_1\leq d_0,h^1-1.$$
This is not transitive since $\II_0\preceq \II_1$ and $\II_1\preceq \IV_1$, but $\II_0\not\preceq \IV_1$.

The \emph{weak} relation, on the other hand, is obtained from the full polarized relation by ignoring subscripts.  If $h^1\geq 3$ it is given simply by $$\I\preceq \II,\III,\IV,\;\;\;\;\;\II\preceq \III,\IV,\;\;\;\text{and}\;\;\;\;\III\preceq \IV,$$ which is obviously a poset with linear order.  If $h^1= 2$ then $\II \not\preceq \III$; while if $h^1=1$ then $\III$ doesn't occur and $\II\not\preceq \IV$.  These observations are generalized in \S\ref{S2.3}.
\end{exa}

Returning to the general case, we used SageMath to compute the convex hull of all polarized relations in $(a_0,b_0,c_0,d_0,a_1,b_1,c_1,d_1)$-space for each $\uh$ with $h^0\leq 6$ and $h^1\leq 18$.  This appears to stabilize\footnote{Indeed, all 15 inequalities already appear if we consider the convex hull of the relations in $(a_0,b_0,c_0,d_0,a_1,b_1,c_1,d_1,h^0,h^1)$-space for $h^0\leq 1$ and $h^1\leq 3$.} and to be described by the following system of 15 inequalities:
\[\begin{array}{lll} a_0-a_1\geq 0 &
a_0-a_1+b_0-b_1\geq 0 &
a_0-a_1+b_0-b_1+c_0-c_1 \geq 0 \\ 
2c_1+d_0-d_1\geq 0 & 
b_0-b_1+c_0-c_1+d_0\geq 0 &
a_0-a_1+2(b_0-b_1)+d_0-d_1\geq 0 \\
2b_0-b_1+d_0-d_1\geq 0 &
b_0-b_1+c_1+d_0-d_1\geq 0 & 
2a_0+4b_0+2c_0+d_0+d_1\leq 2h^1\\ 
2(c_0-c_1)+d_0-d_1\geq 0 &
a_1+2b_0+c_1+d_0\leq h^1 & 
a_0-a_1+2(c_0-c_1)+d_1\geq 0 \\
c_0-c_1+d_1\geq 0 & 
2(b_0-b_1)+2d_0-d_1\geq 0 & 
a_0+b_0+b_1+c_0+d_1\leq h^1.
\end{array}\]
That is, satisfying these inequalities is a necessary condition to have $\DI_1\preceq \DI_0$ (in $\R(\uh)$).  It also appears to be a sufficient condition to have $\widetilde{\DI}_1\preceq \widetilde{\DI}_0$ (in $\R_0^{\circ}(\uh)$), but unfortunately not for $\DI_1\preceq \DI_0$ (which is somewhat mystifying).  Again, the statements in this paragraph have only been experimentally verified.

At least for $\R_0^{\circ}(\uh)$, the situation improves dramatically as soon as $h^1\geq 3h^0$.  As we shall see in the next subsection, the first three inequalities alone will then completely describe the weak polarized relations.

\subsection{Saturation for weak polarized relations graphs}\label{S2.3}

We conclude with some results of a more general flavor for arbitrary weight.  One easy observation (related to double suspensions of singularities) is that if $\uh=(h^0,\uht,h^0)$ (in weight $r$) for some $\uht$ (in weight $r-2$), then \emph{$\R_{k-1}^{\circ}(\uht)$ is the same graph as $\R^{\circ}_{k}(\uh)$}.  So it suffices to study $\R_0^{\circ}(\uh)$.

Working in weight $r$, assume that $h^0>0$.  Each vertex of $\R^{\circ}_0(\uh)$ corresponds to the $\gr_F^0$ part $\widetilde{\DI}$ of an admissible degeneration $\DI$ for the Hodge numbers $\uh$.  Let $\A^r(m)$ be the set of functions $\ua\colon \{0,1,\ldots,r\}\to \ZZ_{\geq 0}$ with $\sum_{k=0}^r \ua(k)=m$, and write $\uA(\ell):=\sum_{k=0}^{\ell} \ua(k)$.  Note that $\A^r(m)$ has a natural poset (hence directed graph) structure given by
$$\ua_1 \sqsubseteq \ua_0 \iff \uA_1(\ell)\leq \uA_0(\ell)\;\;(\forall \ell).$$
The process of taking LMHS does not change the ranks of the $\{\gr_F^i\}$, and cannot decrease the ranks of the $\{\gr_F^0\cap W_j\}$.  Thus the vertices of $\R^{\circ}_0(\uh)$ are a subset of $\A^r(h^0)$, and its arrows are a subset of those induced by $\sqsubseteq$.

\begin{defn}
$\R^{\circ}_0(\uh)$ is \emph{maximal} if its vertex set equals $\A^r(h^0)$, and \emph{saturated} if it is maximal and its edge set equals $\sqsubseteq$.
\end{defn}

Define a function $\ub^r$ on $\{0,1,\ldots,r\}$ by $\ub^r(i)=2i+1=\ub^r(r-i)$ for $i<\tfrac{r}{2}$, and $\ub^r(\tfrac{r}{2})=r$ for $r$ even. So for instance $\ub^5=(1,3,5,5,3,1)$ and $\ub^6=(1,3,5,6,5,3,1)$.  Also define $\uc^r$ by $\uc^r(0)=\uc^r(r)=1$ and $\uc^r(i)=2$ for $0<i<r$.  For functions $\underline{f}$ and $\underline{g}$ on $\{0,1,\ldots,r\}$, we write $\underline{f}\geq \underline{g}$ iff $\underline{f}(i)\geq \underline{g}(i)$ ($\forall i$) and denote scalar multiplication (of the entire function) by $e.\underline{f}$.

\begin{prop}\label{p19}
$\R^{\circ}_0(\uh)$ is \textup{(i)} maximal iff $\uh\geq h^0.\uc^r$ and \textup{(ii)} saturated iff $\uh\geq h^0.\ub^r$.
\end{prop}

\begin{proof}
We need the following language:  recall that an \emph{$N$-string} of length $\ell$ starting at $(p,q)$ is a $\CC$-MHS with rank-1 pieces of types $(p,q),(p+1,q+1),\ldots,(p+\ell,q+\ell)$.  Similarly, a \emph{$\nabla$-string} of length $\ell$ starting at $(p,q)$ is a pure $\CC$-HS with rank-1 pieces of types $(p,q),(p+1,q-1),\ldots,(p+\ell,q-\ell)$.\footnote{The terminology reflects the operation of $\overline{\nabla}_{t\frac{d}{dt}}$ for a VHS, and of the monodromy logarithm $N=-2\pi\ay\mathrm{Res}_0(\nabla)$ in the limit.  The strings can be thought of as irreducible representations of $\mathrm{SL}_2$.}  We can also speak of strings ``centered at'' a midpoint $(p,q)$, where $p$ and $q$ are integers for $\ell$ even and half-integers for $\ell$ odd.

Since all MHSs for us are LMHSs, there is an action of $N$ whose $j^{\text{th}}$ powers produce isomorphisms between $\gr^W_{r+j}$ and $\gr^W_{r-j}$.  By \cite[Thm.~5.18]{KPR}, the allowable degenerations are sums of the following ``atomic'' ones.  First, we select a $\nabla$-string in an $N$-primitive part of weight $w\leq r$, centered at $(p,q)$.  The degeneration will replace it by an $N$-string of the same length with the same center.  However, we must \emph{simultaneously} degenerate $\nabla$-strings of the same length centered at $(q,p)$, as well as $(p+1,q+1)$, $(q+1,p+1)$ etc.~up to $(p+r-w,q+r-w)$ and $(q+r-w,p+r-w)$.\footnote{These $\nabla$-strings are ones which are sent to the $\nabla$-strings in weight $w$ by powers of $N$.}

Here we are only \emph{a priori} interested in $\nabla$-strings which start at $(0,k)$ for some $k\in \{0,\ldots,r\}$.  If the string has length $\ell$, then the degeneration subtracts $1$ from this spot and adds $1$ to the $(0,k-\ell)$ spot.  Such a degeneration is allowable if and only if the primitive part in weight $k$ (of the MHS undergoing degeneration) contains the direct sum of the length-$\ell$ string \emph{and} its complex conjugate.  But then the MHS must contain a \emph{block} where this direct sum is repeated in weights $k+2,k+4,\ldots,2r-k$.  Taking ranks of $\gr_F^i$'s of this block produces a lower bound on $\uh$ which must be satisfied for the ``atomic'' degeneration to be possible.  

For each $k$, the largest such bound occurs when $\ell=k-1$, so that the sum of primitive strings in weight $k$ is $(1,2,2,\ldots,2,1)=\uc^r$.  If we are degenerating the initial pure HS then $k=r$ and this \emph{is} the bound.  For $k<r$, the bound is given by summing $r-k+1$ shifts of $\uc^k$, which is maximized by taking $k=\tfrac{r}{2}$ resp.~$\tfrac{r+1}{2}$ (for $r$ even resp.~odd), yielding $\ub^r$.  To see that the bounds of $h^0.\uc^r$ resp.~$h^0.\ub^r$ suffice in general, we simply take sums of atomic degenerations.  For necessity, writing $\underline{e}^i$ for ``standard basis functions'', we degenerate to $\ua=h^0.\underline{e}^1$ from $\ua=h^0.\underline{e}^r$ for (i) and $\ua=h^0.\underline{e}^{\lfloor\frac{r}{2}\rfloor}$ for (ii).
\end{proof}

\begin{rem}
While there are plenty of VHSs in algebraic geometry with Hodge numbers $\uh$ that \emph{do not} satisfy the bound in Prop.~\ref{p19}(ii), we want to point out that (at least) two particularly natural ones \emph{do}.  Let $\widetilde{U}_{n,d}\subset \PP H^0(\PP^n,\mathcal{O}(d))$ be the complement of the discriminant locus, with quotient $U_{n,d}$ by $\mathrm{PGL}_{n+1}$, and consider the VHSs:
\begin{itemize}[leftmargin=0.5cm]
\item $\mathcal{V}_{n,d}$ over $U_{n,d}$ given by the middle primitive cohomologies (weight $r=n-1$) of hypersurfaces of degree $d\geq n+1$; and
\item $\mathcal{V}'_{n,d}$ over $U_{n,2d}$ given by the middle primitive cohomologies (weight $r=n$) of double covers of $\PP^n$ branched along hypersurfaces of degree $2d$ (with $d\geq n+1$). 
\end{itemize}
(Note that the VHS in Example \ref{Ex5} is $\mathcal{V}'_{2,4}$.) In both cases, the Hodge numbers satisfy $\uh\geq h^0.\ub^r$.  In fact, something much stronger is true:  we have 
\begin{equation}\label{eIN}
h^1\geq rh^0
\end{equation}
in both cases.  Since the Hodge numbers are unimodal, this is enough.

To see this, note that for hypersurfaces [resp.~double covers] $h^p$ identifies with the coefficient of $t^{(p+1)d-n-1}$ [resp.~$t^{(2p+1)d-n-1}$] in $(1+t+t^2+\cdots+t^{d-2})^{n+1}$ [resp.~$(1+t+t^2+\cdots+t^{2d-2})^{n+1}$], see \cite[Cor.~17.5.4]{Ar}.  The unimodality follows immediately from hard Lefschetz applied to $H^*((\PP^{d-2})^{\times (n+1)})$ [resp.~$H^*((\PP^{2d-2})^{\times (n+1)})$].  For the inequality, we compute that $h^0=\binom{d-1}{n}$ (in both cases) and $h^1=\binom{2d-1}{n}-(n+1)\binom{d}{n}$ [resp.~$\binom{3d-1}{n}-(n+1)\binom{d}{n}$].  So in the hypersurface case, \eqref{eIN} reduces to the statement that
$$f_n(d):=\frac{h^1}{h^0}-(n-1)=\prod_{j=1}^n\left(1+\frac{d}{d-j}\right)-(n+1)\left(1+\frac{n}{d-n}\right)-n+1$$
is nonnegative for $d\geq n+1\geq 3$.  The double-cover case immediately follows from this.  

It is easy to check that $f_n(n+1)\geq 2^{n-1}(n+1)-n(n+3)>0$ for $n\geq 3$, and that $f_3(d),f_4(d)\geq 0$ for $d>3$.  Since $\lim_{d\to \infty}f_n(d)=2^n-2n\geq 0$, it now suffices to show $$\textstyle(x-n)^2f_n'(x)=-\left(\prod_{j=1}^n(1+\tfrac{x}{x-j})\right)\sum_{j=1}^n j\tfrac{(x-n)^2}{(x-j)(2x-j)}+n^2+n\leq 0$$
for $n\geq 5$ and $x\geq n+2$.  Indeed, the terms of the sum are increasing on $x\geq n+2$ and the product is $>2^n$; substituting $n+2$ in the sum and using $-2^n+n^2+n<0$ (for $n\geq 5$) gives the result.
\end{rem}

\end{document}